\documentclass[draft,12pt]{amsart}%

\usepackage{amscd,amsmath,amsthm,amssymb}
\usepackage[left]{lineno}
\usepackage{color}
\usepackage{stmaryrd}
\usepackage[utf8]{inputenc}
\usepackage{cleveref}
\usepackage{epstopdf}
\usepackage{graphicx}
\usepackage{xcolor}

\usepackage{tikz}

\definecolor{verylight}{gray}{0.97}
\definecolor{light}{gray}{0.9}
\definecolor{medium}{gray}{0.85}
\definecolor{dark}{gray}{0.6}

\def\NZQ{\mathbb}               
\def\NN{{\NZQ N}}

\def\ZZ{{\NZQ Z}}

\renewcommand{\qedsymbol}{$\square$} 

\def\G{{\mathcal G}}

\def\D{{\mathcal D}}

\def\ab{{\mathbf a}}
\def\bb{{\mathbf b}}

\def\wb{{\mathbf w}}

\def\0b{{\mathbf 0}}

\def\reg{{\mathbf reg}}
\def\height{\operatorname{ht}}
\def\depth{\operatorname{depth}}
\def\opn#1#2{\def#1{\operatorname{#2}}} 
\opn\chara{char} \opn\length{\ell} \opn\pd{pd} \opn\rk{rk}
\opn\projdim{proj\,dim} \opn\injdim{inj\,dim} \opn\rank{rank}
\opn\depth{depth} \opn\grade{grade} \opn\height{height}
\opn\embdim{emb\,dim} \opn\codim{codim}

\opn\Tr{Tr} \opn\bigrank{big\,rank}
\opn\superheight{superheight}\opn\lcm{lcm}
\opn\trdeg{tr\,deg}
	\opn\reg{reg} \opn\lreg{lreg} \opn\ini{in} \opn\lpd{lpd}
	\opn\size{size} \opn\sdepth{sdepth}
	\opn\link{link}\opn\fdepth{fdepth}\opn\lex{lex}
	\opn\tr{tr}
	\opn\type{type}
	\opn\gap{gap}
	\opn\arithdeg{arith-deg}
	\opn\HS{HS}
	\opn\GL{GL}
	\opn\div{div} \opn\Div{Div} \opn\cl{cl} \opn\Cl{Cl}
	\opn\Spec{Spec} \opn\Supp{Supp} \opn\supp{supp} \opn\Sing{Sing}
	\opn\Ass{Ass} \opn\Min{Min}\opn\Mon{Mon}
	\opn\Ann{Ann} \opn\Rad{Rad} \opn\Soc{Soc}\opn\Deg{Deg}
	\opn\Im{Im} \opn\Ker{Ker} \opn\Coker{Coker} \opn\Am{Am}
	\opn\Hom{Hom} \opn\Tor{Tor} \opn\Ext{Ext} \opn\End{End}
	\opn\Aut{Aut} \opn\id{id}
	
	\opn\nat{nat}
	\opn\pff{pf}
	\opn\Pf{Pf} \opn\GL{GL} \opn\SL{SL} \opn\mod{mod} \opn\ord{ord}
	\opn\Gin{Gin} \opn\Hilb{Hilb}\opn\sort{sort}
	\opn\PF{PF}\opn\Ap{Ap}
	\opn\mult{mult}
	\opn\bight{bight}
	\opn\aff{aff}
	\opn\relint{relint} \opn\st{st}
	\opn\lk{lk} \opn\cn{cn} \opn\core{core} \opn\vol{vol}  \opn\inp{inp} \opn\nilpot{nilpot}
	\opn\link{link} \opn\star{star}\opn\lex{lex}\opn\set{set}
	\opn\width{wd}
	\opn\Fr{F}
	\opn\QF{QF}
	\opn\G{G}
	\opn\type{type}\opn\res{res}
	\opn\conv{conv}
	\opn\Ind{Ind}
	\opn\gr{gr}
	
	\def\pot#1#2{#1[\kern-0.28ex[#2]\kern-0.28ex]}

	\opn\dirlim{\underrightarrow{\lim}}
	\opn\inivlim{\underleftarrow{\lim}}
	\let\to=\rightarrow
	
	\def\Implies{\ifmmode\Longrightarrow \else
		\unskip${}\Longrightarrow{}$\ignorespaces\fi}
	\def\implies{\ifmmode\Rightarrow \else
		\unskip${}\Rightarrow{}$\ignorespaces\fi}
	\def\iff{\ifmmode\Longleftrightarrow \else
		\unskip${}\Longleftrightarrow{}$\ignorespaces\fi}

	\let\:=\colon
	\newtheorem{Theorem}{Theorem}[section]
	\newtheorem{Lemma}[Theorem]{Lemma}
	\newtheorem{Corollary}[Theorem]{Corollary}

	\newtheorem{Example}[Theorem]{Example}
	
	\newtheorem{Definition}[Theorem]{Definition}

	\let\epsilon\varepsilon
	\let\kappa=\varkappa
	\def\qed{\ifhmode\textqed\fi
		\ifmmode\ifinner\quad\qedsymbol\else\dispqed\fi\fi}
	\def\textqed{\unskip\nobreak\penalty50
		\hskip2em\hbox{}\nobreak\hfil\qedsymbol
		\parfillskip=0pt \finalhyphendemerits=0}
	\def\dispqed{\rlap{\qquad\qedsymbol}}
	
	\opn\dis{dis}
	\def\pnt{{\raise0.5mm\hbox{\large\bf.}}}
	
	\opn\Lex{Lex}
	
\begin{document}
		\title {Cohen-Macaulayness of Powers of Edge Ideals of Weighted Oriented Graphs}

\author{Truong Thi Hien}
\address{Faculty of Natural Sciences,   Hong  Duc University,  No. 565 Quang Trung, Dong Ve, Thanh Hoa, Vietnam}
\email{hientruong86@gmail.com}

\author{Jiaxin Li}
\address{School of Mathematical Sciences, Soochow University, Suzhou, Jiangsu, 215006, P. R.~China}
\email{lijiaxinworking@163.com}

\author{Tran Nam Trung}
\address{Institute of Mathematics, Vietnam Academy of Science and Technology, 18 Hoang Quoc Viet, 10072 Hanoi, Vietnam}
\email{tntrung@math.ac.vn}

\author{Guangjun Zhu$^{\ast}$}
\address{School of Mathematical Sciences, Soochow University, Suzhou, Jiangsu, 215006, P. R.~China}
\email{zhuguangjun@suda.edu.cn}

		 \thanks{ * Corresponding author.}
		


		
		

\begin{abstract} For the edge ideal $I(\D)$ of a weighted oriented graph $\D$, we prove that its symbolic powers $I(\D)^{(t)}$ are Cohen-Macaulay for all $t\geqslant 1$ if and only if the underlying graph $G$ is composed of a disjoint union of some complete graphs. We also completely characterize  the Cohen-Macaulayness of the ordinary powers $I(\D)^t$ for all $t\geqslant 2$. Furthermore,  we  provide a criterion for  determining whether $I(\D)^t=I(\D)^{(t)}$.
\end{abstract}
		\setcounter{tocdepth}{1}
		
		\maketitle
\section{Introduction}

An oriented graph $\D=(V(\D), E(\D))$ consists of a simple underlying graph $G$ in which each edge is oriented, i.e., it is a directed graph with no multiple edges or loops. The elements of $E(\D)$ are denoted by ordered pairs to reflect the orientation. For example,
$(u, v)$ represents an edge directed from $u$ to $v$. A {\it vertex-weighted} (or simply  weighted) oriented graph $\D$ is a graph  equipped with a weight function $\omega \colon V(\D)\to \ZZ_{>0}$. The pair  $(\D,\omega)$ is called a weighted oriented graph. When there is no confusion, we will simply use $\D$ to  represent  this pair.

Let $(\D,\omega)$ be a weighted oriented graph with an underlying graph $G$ and  a vertex set $V(\D)=\{1,2,\ldots,n\}$. Let $R=K[x_1, \ldots, x_n]$ be a polynomial ring  with $n$ variables  over a field $K$. The edge ideal of $\D$ is defined as
$$I(\D) = (x_ix_j^{\omega(j)} \mid (i,j)\in E(\D)).$$
In particular, if $\omega(j)=1$ for all $j\in V(\D)$, then $I(\D)=I(G)$.

The aim of this paper is to characterize the Cohen-Macaulayness of the symbolic powers of the edge ideal $I(\D)$ in terms of $\D$'s structure. Recall that the $t$-th symbolic power $I^{(t)}$ of an ideal $I$ in $R$ is defined as the intersection of the primary
components of $I^t$ associated with the minimal primes.

If $I$ is the Stanley-Reisner ideal of a simplicial complex $\Delta$, Terai and Trung \cite{TT} proved that $I^{(t)}$ is Cohen-Macaulay for some (or for all) $t \geqslant 3$ if and only if $\Delta$ is a matroid. In this paper, we investigate this property for the ideal $I(\D)$. However, $I(\D)$ is not square-free,  therefore,  we cannot directly apply the linear programming technique used to  study the Cohen-Macaulayness of $I(\D)^{(t)}$ as being done for  square-free monomial ideals (see, for example,  \cite{MT,TT}). Fortunately, $I(\D)$ has a nice primary decomposition, as shown in \cite{PRT}. Using the  Hochster formula for the depth of a monomial ideal, we can prove that if $I(\D)^{(t)}$ is Cohen-Macaulay for all $t\geqslant 1$, then the independence complex $\Delta(G)$ of the underlying graph $G$ of $\D$ is a matroid, and  $G$ is a disjoint union of complete graphs. The idea of the proof is as follows: If $\dim \Delta(G)=1$, then $\Delta(G)$ is a matroid.  We prove this by studying the integer solutions of  certain linear inequalities. For higher dimensions, we first prove that $\Delta(G)$ is locally a matroid and then use a result from \cite{TT} to show that $\Delta(G)$ is a matroid. Namely,

\medskip

\noindent {{\bf Theorem 1} (see Theorem \ref{cmPowers}). \it Let $\D$ be a weighted oriented graph with the underlying graph $G$. Then $I(\D)^{(t)}$ is Cohen-Macaulay for all $t\geqslant 1$ if and only if $G$ is a disjoint union of complete graphs.
}

\medskip

In contrast to square-free monomial ideals, as discussed in \cite{TT}, for every integer $m\geqslant 1$, there exists  a weighted oriented graph $\D$ with $4$ vertices such that $I(\D)^{(t)}$ is Cohen-Macaulay if and only if $t \leqslant m$ (see Example \ref{E4}).

\medskip

For ordinary powers,  Terai and Trung in \cite{TT} proved that if $I$ is a square-free monomial ideal, then $I^t$ is Cohen-Macaulay for some (or for all) $t \geqslant 3$ if and only if $I$ is a complete intersection. In fact, we fully  characterize  the Cohen-Macaulayness of $I(\D)^t$ for each $t\geqslant 2$. The basic tool is the criterion for the equality $I(\D)^t = I(\D)^{(t)}$ and it is also interesting in itself. This criterion was  obtained for the edge ideal $I(G)$ (see \cite{RTY}), and for $I(\D)$ with $t=2$ (see \cite{GMV}). We generalize these results to $I(\D)$ for any $t$. Before stating the result, we need to define some terms.  A {\em sink} vertex of $\D$ is a vertex with only incoming edges. A {\em source} vertex of $\D$ is a vertex with only outgoing  edges. In this paper, we always assume that if $v$ is a source  vertex, then $\omega(v) = 1$. This convention clearly  does not change the ideal $I(\D)$. Let $V^{+}(\D)=\{v\in V(D) \mid \omega(v) \geq 2 \}$. Then we have:

\medskip

\noindent {{\bf Theorem 2} (see Theorem \ref{equal}). Let $\D$ be a weighted oriented graph with the underlying graph $G$. For any $t\geqslant 2$, the following conditions are equivalent:
\begin{itemize}
\item[(1)] $I(\D)^t=I(\D)^{(t)}$.
\item[(2)] Every vertex in $V^{+}(\D)$ is a sink, and $G$ contains no odd cycles of length $2s - 1$ for any $2 \leqslant s \leqslant t$.
\end{itemize}
}

\medskip

This result plays a key role in characterizing the Cohen-Macaulayness of $I(\D)^t$ for any $t\geqslant 2$. In fact, by using it we can reduce our problem to studying the Cohen-Macaulayness of $I(G)^t$. The result for $t\geqslant 3$ is the following theorem.

\medskip

\noindent {{\bf Theorem 3} (see Theorem \ref{cmPowers}). Let $\D$ be a weighted oriented graph with the underlying graph $G$. Then the following conditions are equivalent:
\begin{itemize}
\item[(1)] $I(\D)^t$ is Cohen-Macaulay for every $t\geqslant 1$.
\item[(2)] $I(\D)^t$ is Cohen-Macaulay for some $t\geqslant 3$.
\item[(3)] $G$ is a disjoint union of edges.
\end{itemize}
}

\medskip

Recall that a well-covered graph is a graph for which every minimal vertex cover has the same size. A well-covered graph $G$ is  a member of the class $W_2$  if $G\setminus v$ is well-covered and  $\alpha(G\setminus v) = \alpha(G)$ for every vertex $v$ of $G$. Then,

\medskip

\noindent {{\bf Theorem 4} (see Theorem \ref{cmPower2}). Let $\D$ be a weighted oriented graph with the underlying graph $G$. Then, $I(\D)^2$ is Cohen-Macaulay if and only if:
\begin{itemize}
\item[(1)] Every vertex in $V^+(\D)$ is a sink, and
\item[(2)] $G$ is a triangle-free graph in the class $W_2$.
\end{itemize}
}

The paper is organized as follows:  Section  \ref{sec:prelim} introduces some  basic facts and properties of simplicial complexes, and symbolic powers of the edge ideal of a weighted oriented graph. It also recalls Hochster's formulas for depth and Betti numbers. In Section  \ref{sec:symbolic powers}, we  deal with the Cohen-Macaulayness of all symbolic powers of the edge ideal of a weighted oriented graph. In Section  \ref{sec:equality},  we study the Cohen-Macaulayness of each ordinary power of such an ideal.

\section{Preliminaries}
 \label{sec:prelim}    
 
Throughout this paper, let $K$ be an arbitrary field and $[n]$ be the set $\{1,\ldots,n\}$. Let $\Delta$ be a {\it simplicial complex} with the vertex set $V(\Delta) = [n]$. Thus $\Delta$ is a collection of subsets of $[n]$ such that if $\G\in\Delta$ and $F\subseteq G$, then $F\in\Delta$. Each element $F\in \Delta$ is called a face of $\Delta$. The dimension of a face $F$ is $|F|-1$.  Define the dimension of $\Delta$ to be $\dim \Delta = d-1$, where $d = \max\{|F| : F\in \Delta\}$. 
 A {\em facet} is a maximal face of $\Delta$ with respect to inclusion. Let $\mathcal{F}(\Delta)$ denote the set of facets of $\Delta$. If all facets of $\Delta$ have the same size, then $\Delta$ is pure.

We define the  {\it Stanley-Reisner} ideal $I_{\Delta}$ of $\Delta$ as the squarefree monomial ideal
$$I_{\Delta} = (x_{j_1} \cdots x_{j_i} \mid j_1  <\cdots < j_i \ \text{ and } \{j_1,\ldots,j_i\} \notin \Delta) \ \text{ in } R = K[x_1,\ldots,x_n]$$
and the {\it Stanley-Reisner} ring of $\Delta$  as  the quotient ring $K[\Delta] = R/I_{\Delta}$. 
We say that $\Delta$ is Cohen-Macaulay (resp. Gorenstein) over $K$ if $K[\Delta]$ has the same property. It is well known that if $\Delta$ is Cohen-Macaulay, then it is pure.

\begin{Lemma}{\em (\cite[Corollary 8.1.7]{HH})}\label{cmcomplex} Every Cohen–Macaulay simplicial complex is connected.
\end{Lemma}

For every face $G$ in $\Delta$, we define
$\lk_{\Delta} (G) = \{F \setminus G \in \Delta \mid G \subseteq F \in \Delta\}$. We  call  this subcomplex  the {\it link} of $G$ in $\Delta$.
A simplicial complex $\Delta$ is called a {\em matroid complex} if it satisfies the {\it exchange property}: if $F$ and $H$ are two faces of $\Delta$ and $F$ has more elements than $H$, then there exists an element in $F$ which is not in $H$ and, when added to $H$, still forms  a face of $\Delta$. We say that $\Delta$ is {\it locally a matroid} if $\lk_\Delta(i)$ is a matroid complex for every vertex $i$ of $\Delta$.

\begin{Lemma} \cite[Theorem 2.7]{TT} \label{locally-matroid} Let $\Delta$ be a simplicial complex with $\dim \Delta \geqslant 2$. Then $\Delta$ is a matroid if and only if it is connected and locally a matroid.
\end{Lemma}

Next, we will review some notation and terminology from graph theory.  Let $G$ be
a graph. We use the symbols $V(G)$ and $E(G)$ to denote the vertex and  edge sets of $G$, respectively. If $S$ is a subset of $V(G)$, then $G[S]$ is the induced subgraph of $G$ on $S$, and  $G\setminus S$ is the induced subgraph of $G$ on $V (G)\setminus S$.
Two vertices in $G$  are adjacent if they share a common edge, and two distinct adjacent vertices are neighbors. The set of neighbors of a vertex $v$ in $G$
is denoted  $N_G(v)$. For a subset $S\subseteq  V(G)$,   we denote its neighbors by
$$N_G (S) = \{x \in V(G)\setminus S \mid N_G(x) \cap S \neq \emptyset\}.$$
The closed neighbors of $S$ are denoted $N_G[S] = S \cup N_G(S)$, and the localization of $G$ with respect to $S$ is denoted by $G_S = G\setminus N_G[S]$. An independent set in $G$ is a set of vertices in which no two vertices are adjacent to each other.  The {\em independence number} of $G$, denoted by $\alpha(G)$,  is the largest cardinality of its maximal independent sets. The set of all independent sets of $G$ is called the independence complex of $G$ and is denoted by $\Delta(G)$. Obviously, $\dim(\Delta(G))=\alpha(G)-1$.

A vertex cover of a graph $G$ is a set of vertices that includes at least one endpoint of each edge in $G$, and a vertex cover is minimal if it is the smallest possible set that satisfies this condition. In this paper, we denote the set of minimal vertex covers of $G$ by $\Gamma(G)$. The {\em covering number} of $G$, denoted by $\beta(G)$, is the  smallest cardinality of its minimal  vertex covers.
A graph $G$ is  {\it well-covered} if every minimal vertex cover of $G$ is of  size  $\beta(G)$. Since the complement of a vertex cover is an independent set,  $G$ is well-covered if and only if every maximal independent set of $G$ is of  size  $\alpha(G)$. A well-covered graph G is said to be a member of the class $W_2$ if $G\setminus v$ is well-covered and $\alpha(G\setminus v) = \alpha(G)$ for every vertex $v$.

\begin{Lemma} {\em (\cite[Lemma 1]{FHN})}\label{well-covered}  Let $G$ be a well-covered graph. Then, for every $S\in\Delta(G)$, $G\setminus N_G[S]$ is a well-covered graph with $\alpha(G\setminus N_G[S])=\alpha(G)-|S|$.
\end{Lemma}

\begin{Lemma} \label{complete-graph} If $\Delta(G)$ is a matroid, then $G$ is a disjoint union of complete graphs.
\end{Lemma}
\begin{proof}  Assume by contradiction that $G$ is not a disjoint union of complete graphs. Then $G$ has three vertices, say $u,v$ and $w$ such that $uv,uw \in E(G)$, but $vw\notin E(G)$. Let $S = \{u,v,w\}$ and $H = G[S]$. Then, $\Delta(H) = \{F \in \Delta(G)| F \subseteq S\}$ is a pure simplicial complex by \cite[Proposition 3.1]{S}. This means that the graph $H$ is well-covered,  which is a contradiction. Therefore, $G$ must be a disjoint union of complete graphs.
\end{proof}

Assume that $V(G)=[n]$. The edge ideal of $G$ is the monomial ideal
$$I(G) = (x_ix_j\mid \{i,j\}\in E(G)).$$
It is well known that $I(G)= I_{\Delta(G)}$ and therefore
\[
\Ass(R/I(G)) = \{(x_j\mid j \in C )\mid  C \in \Gamma(G)\}.
\]

Given a  weighted oriented graph $(\D,\omega)$, $H$ is called to be an induced subgraph of $(\D,\omega)$ if $V(H) \subset V(\D)$, and for any $u,v \in V(H)$, $uv \in E(H)$ if and only if $uv \in E(\D)$. Furthermore, the weight $\omega_H(x)$ of vertex $x$ in $H$ is equal to its weight $\omega_{\D}(x)$ in $\D$.
For $x\in V(\D)$, the sets $N_{\D}^{+}(x) = \{y \mid (x, y) \in E(\D)\}$ and $N_{\D}^{-}(x) = \{y \mid (y, x) \in E(\D)\}$ are called the {\em out-neighborhood} and the {\em in-neighborhood} of $x$, respectively. Furthermore, the {\em neighborhood} of $x$ is the set  $N_{\D}(x) = N_{\D}^{+}(x) \cup N_{\D}^{-}(x) \text{ and } N_{\D}[x]=\{x\}\cup N_{\D}(x)$. Clearly,  $N_{\D}(x)=N_G(x)$ and $N_{\D}[x] = N_G[x]$.

Let $C$ be a vertex cover of $\D$. Define
\begin{align*}
    L_1(C) &= \{x\in C\mid (x,y)\in E(\D) \text{ for some } y\notin C\},\\
    L_2(C) &=\{x\in C\setminus L_1(C) \mid (y,x)\in E(\D) \text{ for some } y\notin C\}, \ \text{ and}\\
    L_3(C) &= \{x\in C\mid N_G(x) \subseteq C\}.
\end{align*}

A vertex cover $C$ of $G$ is called a {\it strong vertex cover}
of $\D$ if either $C$ is a minimal vertex cover of $G$ or, for all $x \in L_3(C)$, there is a $(y, x) \in E(\D)$ such that $y \in L_2(C)\cup L_3(C)$ with $\omega(y) \geqslant 2$. The set of strong vertex covers of $\D$  is denoted by $\Gamma(\D)$. Clearly,  $\Gamma(G)\subseteq \Gamma(\D)$, and if  $C\in\Gamma(\D)$, then   $C\in\Gamma(G)$ if and only if $L_3(C)=\emptyset$ .

 For any $C\in\Gamma(\D)$, let
$$I_C= (x_i, x_j^{\omega(j)} \mid i \in L_1(C), \ j \in C\setminus L_1(C)).$$
Then $I(\D)$ has a minimal primary decomposition as follows.

\begin{Lemma}{\em (\cite[Theorem 25]{PRT})}\label{decomposition} The minimal primary decomposition of $I(\D)$ is given by 
$$I(\D) = \bigcap_{C \in \Gamma(\D)} I_C.$$
\end{Lemma}

For an ideal $I \subset R$ and any integer $t\ge 1$, the {\em $t$-th symbolic power} of $I$ is defined as
		\[
		I^{(t)} = \bigcap_{p \in \operatorname{Min}(I)} I^t R_p \cap R,
		\]
		where $\operatorname{Min}(I)$ is the set of minimal primes of $I$.
In the case that $I$ is a monomial ideal with a minimal primary decomposition
$$I = Q_1\cap \cdots\cap Q_r\cap Q_{r+1}\cap\cdots \cap Q_s,$$
where each $Q_i$ is a monomial primary ideal and
$$\operatorname{Min}(I) = \{\sqrt{Q_i}\mid i=1,\ldots,r\},$$
then
$$I^{(t)} = Q_1^t \cap \cdots\cap Q_r^t.$$

Since $\sqrt{I(\D)} = I(G)$, we get $\Min(I(\D)) = \Ass(R/I(G))$. Together with Lemma \ref{decomposition}, this  yields:

\begin{Lemma}\label{SPowers} $I(\D)^{(t)} = \bigcap\limits_{C\in\Gamma(G)}I_C^t$.    
\end{Lemma}

To study the Cohen-Macaulayness of a monomial ideal, we need the following lemma.

\begin{Lemma}{\em (\cite[Theorem 7.1]{H})}\label{depth} Let $I$ be a monomial ideal. Then
$$\depth(R/I) = \min\{ \depth(R/\sqrt{I:f}) \mid f \text{ is a monomial such that } f \notin I\}.$$
\end{Lemma}

An ideal $I$ of $R$ is said to be Cohen-Macaulay if $R/I$ is. Using Lemma \ref{depth}, we can derive the following two lemmas.

\begin{Lemma}\label{cm} A monomial ideal $I$ is Cohen-Macaulay if and only if  it is unmixed and $\sqrt{I:f}$ is Cohen-Macaulay for every monomial $f \notin I$.
\end{Lemma}

\begin{Lemma}{\em (\cite[Theorem 2.6]{HTT})}\label{radical} If a monomial ideal $I$ is Cohen-Macaulay, then so is $\sqrt{I}$.
\end{Lemma}

For $\mathbf{a}= (a_1, \ldots, a_n) \in \mathbb{N}^n$, set $x^{\mathbf{a}}=x_1^{a_1}\cdots x_n^{a_n}$. To obtain the simplicial complex of the square-free monomial ideal $\sqrt{I\colon x^{\mathbf{a}}}$,
let
$$\Delta_{\ab}(I) = \{F\subseteq [n] \mid x^{\ab} \notin IR[x_i^{-1}\mid i\in F]\}.$$
This is a simplicial complex called the degree complex of $I$ at degree $\ab$.
        
\begin{Lemma}{\em (\cite[Lemma 2.19]{M})}\label{degree complex2} Let $I\subseteq R$ be a monomial ideal and $\mathbf{a} \in \mathbb{N}^n$. Then
        	\[
        	I_{\Delta_{\mathbf{a}}(I)} = \sqrt{I : x^{\mathbf{a}}}.
        	\]
        	In particular, $x^{\mathbf{a}} \in I$ if and only if $\Delta_{\mathbf{a}}(I)$ is the void complex.
        \end{Lemma}

We conclude this section with the Hochster formula for computing the Betti numbers of monomial ideals. Let $I$ be a monomial ideal of $R$, and assume that $R/I$ has the minimal free $\NN^n$-graded resolution
     	
$$0 \to \bigoplus_{\ab\in\NN^n} R(-\ab)^{\beta_{p,\ab}} \to \bigoplus_{\ab\in\NN^n} R(-\ab)^{\beta_{p-1,\ab}} \to \cdots \to \bigoplus_{\ab\in\NN^n} R(-\ab)^{\beta_{0,\ab}} \to R/I \to 0,$$
where $p=\pd(R/I)$ is the projective dimension of $R/I$, and $R(-\ab)$ is the free module obtained by shifting the degrees in $R$ by $\ab$. The numbers $\beta_{i,\ab}$'s are positive integers  called the $i$-th multigraded Betti numbers of $R/I$ in degree $\ab$. When emphasizing the Betti number of $R/I$, we write $\beta_{i,\ab}(R/I)$ instead of $\beta_{i,\ab}$. Then
$$\depth (R/I) = \min \{n-i \mid \beta_{i,\ab}(R/I) \neq 0 \text{ for some } \ab\in\NN^n\} = n-p.$$
Note that $R/I$ is Cohen-Macaulay if and only if $\depth (R/I)=\dim (R/I)$.

For a  monomial ideal $I$, let $\mathcal{G}(I)$ denote its unique minimal set of monomial generators.

\begin{Lemma}\label{betaneq0} Let $I$ be a monomial ideal. If $\beta_{i,\mathbf{a}}(R/I)\neq 0$ for some $i$, then there exist some  monomials $m_1,\ldots,m_s\in \mathcal{G}(I)$ such that $x^{\mathbf{a}}=\operatorname{lcm}(m_1,\ldots,m_s)$.
\end{Lemma}
\begin{proof}  See, for example, \cite[Exercise 1.2]{MS}.
\end{proof}

Now, given a monomial ideal $I$ and a degree $\mathbf{a} \in \mathbb{N}^n$, define
        	$$K^\mathbf{a}(I) = \{ \text{squarefree vectors } \boldsymbol{\tau}\in \{0,1\}^n| x^{\mathbf{a}-\boldsymbol{\tau}} \in I \}$$
        	to be the (upper) Koszul simplicial complex of $I$ in degree $\mathbf{a}$.

        \begin{Lemma}{\em (\cite[Theorem 1.34]{MS})}\label{betti1}
        	Given a vector $\mathbf{a} \in \mathbb{N}^n$, the Betti number of $R/I$ in degree $\ab$ can be expressed as
        	\[
        	\beta_{i,\mathbf{a}}(R/I) = \dim_K \tilde{H}_{i-2}(K^\mathbf{a}(I); K),
        	\]
        	where $\tilde{H}_{i-2}(K^\mathbf{a}(I); K)$ is the  $(i-2)$-th  reduced homology of $K^\mathbf{a}(I)$ over $K$.
        \end{Lemma}

\section{Cohen-Macaulayness of symbolic powers of edge ideals} \label{sec:symbolic powers}

In this section, we will assume that $\D$ is a weighted oriented graph with an underlying graph $G$ and a vertex set $[n]$. Our goal is to characterize $\D$ such that $I(\D)^{(t)}$ is Cohen-Macaulay for all $t\geqslant 1$.

\begin{Lemma}\label{degree complex3} Let $\D$ be a weighted oriented graph. For any $\ab=(a_1,\ldots,a_n) \in \mathbb{N}^n$ and $t\geqslant 1$, let
$$\mathcal C = \{C\in \Gamma(G)\mid \sum\limits_{i\in L_{1}(C)} a_{i}+\sum\limits_{j\in C\backslash L_{1}(C) } \left\lfloor \frac{a_{j}}{\omega(j)}\right\rfloor \leqslant t-1\}.$$
Then
    $$\sqrt{I(\D)^{(t)} \colon x^{\ab}}=\bigcap_{C\in\mathcal C} (x_i\mid i\in C),$$
and
    $$\mathcal F(\Delta_{\ab}(I(\D)^{(t)}) = \{S\in\mathcal F(\Delta(G)) \mid [n] \setminus S \in \mathcal C\}.$$
\end{Lemma}
	
\begin{proof}
By Lemma \ref{SPowers}, we have
$$I(D)^{(t)} = \bigcap_{C\in {\Gamma(G)}} I_C^t.$$
Now for any  $C\in\Gamma(G)\subseteq\Gamma(\D)$, since
$$I_C = (x_i, x_j^{\omega(j)}\mid i\in L_1(C), \ j\in C\setminus L_1(C)),$$
it follows that $x^{\ab}\notin I_C^t$ if and only if
\begin{equation}\label{EQ1}
\sum\limits_{i\in L_{1}(C)} a_{i}+\sum\limits_{j\in C\backslash L_{1}(C) } \left\lfloor\frac{a_{j}}{\omega(j)}\right\rfloor \leq t-1.
\end{equation}

Note that $\sqrt{I_C^t \colon x^{\ab}} = (x_k\mid k\in C)$ if $x^{\ab}\notin I_C^t$, together with $(\ref{EQ1})$ it yields
$$\sqrt{I(\D)^{(t)} \colon x^{\ab}}=\bigcap_{C\in \mathcal C} (x_i\mid i\in C).$$

The second equality of the lemma follows from this equality and from Lemma \ref{degree complex2}. Thus, the proof is complete.
\end{proof}

\begin{Lemma}\label{a=2}
Let $D$ be a weighted oriented graph with $\alpha(G)=2$, and let $I(D)^{(t)}$ be Cohen-Macaulay for all $t\ge 1$, then $\Delta(G)$ is a matroid.
\end{Lemma}
\begin{proof} Let $I=I(D)$,  then  $I(G) = \sqrt{I^{(t)}}$ for  any $t\geqslant 1$. Together with Lemma \ref{radical}, this implies that $\Delta(G)$ is Cohen-Macaulay, so $\Delta(G)$ is pure and $\dim (\Delta(G)) = 1$. 
Thus  $\Delta(G)$ be regarded as a simple graph with a vertex set $[n]$ and an edge set $\mathcal F(\Delta(G))$. By  Lemma \ref{cmcomplex}, $\Delta(G)$ is connected. For simplicity, we will also denote this graph by $\Delta(G)$ when there is no  confusion. Note that $\Delta(G)$ is a triangle-free graph because if $\Delta(G)$ had a triangle, then the three  vertices on the triangle would be an independent set of $G$, so $\alpha(G) \geqslant 3$, which is a contradiction.

We will now prove that the simplicial complex $\Delta(G)$ is a matroid complex. Suppose by contradiction that $\Delta(G)$ is not a matroid complex. According to the definition of matroid complexes, there are distinct vertices $i$ and $j$ such that  $\{i,j\} \in \Delta(G)$, as well as a vertex $v_0 \in [n]\setminus\{i,j\}$ such that $\{v_0,i\} \notin \Delta(G)$ and $\{v_0,j\} \notin \Delta(G)$. Since $\Delta(G)$ is connected, there exists a path $P = v_0v_1\cdots v_s v_{s+1}$ of the shortest possible length in $\Delta(G)$ with $v_{s+1}\in\{i,j\}$. Without loss of generality, assume that $v_{s+1} = j$. By the minimal length  of $P$, we have $\{v_{s-1}, j\},\{v_{s-1},i\} \notin \Delta(G)$. Furthermore, $\{v_s, i\} \notin \Delta(G)$ because $\Delta(G)$ is triangle-free. Therefore, $\{v_s, i\}$, $\{v_{s-1}, i\}$, and $\{v_{s-1}, j\}$ are edges in $G$. By symmetry,  there are four possible cases depending on the direction of the edges  in $\D$ (see Figure 1):

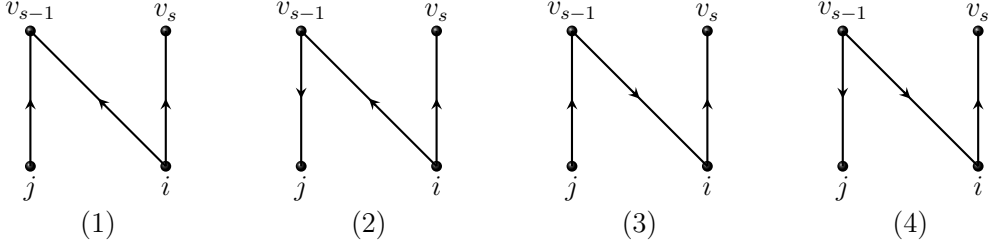
\begin{figure}[htbp]
      	\centering
      	\begin{tikzpicture}[thick, scale=0.9, every node/.style={scale=0.9}]
      		\def\nodeA{(0,2)}    
      		\def\nodeB{(0,0)}    
      		\def\nodeC{(2,0)}    
      		\def\nodeD{(2,2)}    
      		
      		\begin{scope}[shift={(0,0)}]
      			\shade[shading=ball, ball color=black] \nodeA circle (.08) node[above] {${v_{s-1}}$};
      			\shade[shading=ball, ball color=black] \nodeB circle (.08) node[below] {$j$};
      			\shade[shading=ball, ball color=black] \nodeC circle (.08) node[below] {$i$};
      			\shade[shading=ball, ball color=black] \nodeD circle (.08) node[above] {${v_s}$};
      			
      			\draw (0,0) -- (0,2);
      			\draw[->, >=stealth] (0,0.9) -- (0,1);
      			\draw (0,2) -- (2,0);
      		    \draw[->, >=stealth] (1.1,0.9) -- (1,1);
      			\draw (2,0) -- (2,2);
      	    	\draw[->, >=stealth] (2,0.9) -- (2,1);

      			\node[below] at (1,-0.5) {(1)};
      		\end{scope}
      		
      		\begin{scope}[shift={(4,0)}]  
      			\shade[shading=ball, ball color=black] \nodeA circle (.08) node[above] {${v_{s-1}}$};
      			\shade[shading=ball, ball color=black] \nodeB circle (.08) node[below] {$j$};
      			\shade[shading=ball, ball color=black] \nodeC circle (.08) node[below] {$i$};
      			\shade[shading=ball, ball color=black] \nodeD circle (.08) node[above] {${v_s}$};
      			
      			\draw (0,0) -- (0,2);
      			\draw[->, >=stealth] (0,1.1) -- (0,1);
      			\draw (0,2) -- (2,0);
      			\draw[->, >=stealth] (1.1,0.9) -- (1,1);
      			\draw (2,0) -- (2,2);
      			\draw[->, >=stealth] (2,0.9) -- (2,1);
      			
      			\node[below] at (1,-0.5) {(2)};
      		\end{scope}
      		
      		\begin{scope}[shift={(8,0)}]  
      			\shade[shading=ball, ball color=black] \nodeA circle (.08) node[above] {${v_{s-1}}$};
      			\shade[shading=ball, ball color=black] \nodeB circle (.08) node[below] {$j$};
      			\shade[shading=ball, ball color=black] \nodeC circle (.08) node[below] {$i$};
      			\shade[shading=ball, ball color=black] \nodeD circle (.08) node[above] {${v_s}$};
      			
      			\draw (0,0) -- (0,2);
      			\draw[->, >=stealth] (0,0.9) -- (0,1);
      			\draw (0,2) -- (2,0);
      			\draw[->, >=stealth] (0.9,1.1) -- (1,1);
      			\draw (2,0) -- (2,2);
      			\draw[->, >=stealth] (2,0.9) -- (2,1);
      			
      			\node[below] at (1,-0.5) {(3)};
      		\end{scope}
      		
      		\begin{scope}[shift={(12,0)}]  
      			\shade[shading=ball, ball color=black] \nodeA circle (.08) node[above] {${v_{s-1}}$};
      			\shade[shading=ball, ball color=black] \nodeB circle (.08) node[below] {$j$};
      			\shade[shading=ball, ball color=black] \nodeC circle (.08) node[below] {$i$};
      			\shade[shading=ball, ball color=black] \nodeD circle (.08) node[above] {${v_s}$};
      			
      			\draw (0,0) -- (0,2);
      		\draw[->, >=stealth] (0,1.1) -- (0,1);
      		\draw (0,2) -- (2,0);
      		\draw[->, >=stealth] (0.9,1.1) -- (1,1);
      		\draw (2,0) -- (2,2);
      		\draw[->, >=stealth] (2,0.9) -- (2,1);
      			
      			\node[below] at (1,-0.5) {(4)};
      		\end{scope}
      	\end{tikzpicture}
      	\caption{The possible directions in $\D$.}
      \end{figure}

{\it Case $1$}: $(j,v_{s-1}),(i,v_s),(i,v_{s-1})\in E(\D)$. Fix an integer $k\ge \max\{\omega(i),\omega(j)\}$ and set $a_i=a_j=k$, $a_{v_s}=0$, $a_{v_{s-1}}=2k\omega({v_{s-1}})$, $t=2k+1$ and $a_{\ell}=0$ for any $\ell\in[n]\backslash\{i,j,v_{s-1},v_s\}$. Then, the following system of inequalities holds:
           \[
           \begin{cases}
           	\lfloor\frac{a_{v_s}}{\omega({v_s})}\rfloor+\lfloor\frac{a_{v_{s-1}}}{\omega({v_{s-1}})}\rfloor\le t-1, &\\ 	
           	a_i+a_j\le t-1, &\\
        	a_i+\lfloor\frac{a_{v_{s-1}}}{\omega({v_{s-1}})}\rfloor\ge t, &\\
           	\lfloor\frac{a_{i}}{\omega({i})}\rfloor+\lfloor\frac{a_{v_s}}{\omega({v_s})}\rfloor+\lfloor\frac{a_{v_{s-1}}}{\omega({v_{s-1}})}\rfloor\ge t, &\\ 	
           	\lfloor\frac{a_{j}}{\omega({j})}\rfloor+\lfloor\frac{a_{v_s}}{\omega({v_s})}\rfloor+\lfloor\frac{a_{v_{s-1}}}{\omega({v_{s-1}})}\rfloor\ge t. &\\ 	
           \end{cases}
           \]
           By Lemma \ref{degree complex3},   $\{i,j\},\{v_{s-1},v_s\}\in \Delta_{\mathbf{a}}(I^{(t)})$ and  $\{v_s,j\},\{i,\ell\},\{j,\ell\}\notin \Delta_{\mathbf{a}}(I^{(t)})$ for all $\ell\in [n]\backslash\{i,j,v_{s-1},v_s\}$. These conditions imply that $\Delta_{\mathbf{a}}(I^{(t)})$ is disconnected. However, by Lemmas \ref{cm} and \ref{degree complex2},
            $\Delta_{\mathbf{a}}(I^{(t)})$ is a Cohen-Macaulay,  so according to Lemma \ref{cmcomplex},  $\Delta_{\mathbf{a}}(I^{(t)})$ is connected  as $\dim (\Delta_{\mathbf{a}}(I^{(t)}) )> 0$, which is a contradiction. Therefore, $\Delta(G)$ is a matroid complex.

\medskip

{\it Case $2$}: $(v_{s-1},j),(i,v_s),(i,v_{s-1})\in E(\D)$. Set $k=\omega({v_{s-1}})(\omega(i)+1)-\omega(i)$, $a_i=\omega(i)+1$, $a_j=k\omega(j)$, $a_{v_{s}}=\omega({v_s})$, $a_{v_{s-1}}=\omega({v_{s-1}})(\omega(i)+1)$, $t=\omega({v_{s-1}})(\omega(i)+1)+2$ and $a_{\ell}=0$ for any $\ell\in[n]\backslash\{i,j,v_{s-1},v_s\}$. Then, the following system of inequalities holds:
           \[
           \begin{cases}
           	\lfloor\frac{a_{v_s}}{\omega({v_s})}\rfloor+a_{v_{s-1}}\le t-1, &\\ 	
           	a_i+\lfloor\frac{a_{j}}{\omega({j})}\rfloor\le t-1, &\\
           	a_i+a_{v_{s-1}}\ge t, &\\ 	
           	\lfloor\frac{a_{i}}{\omega({i})}\rfloor+\lfloor\frac{a_{v_s}}{\omega({v_s})}\rfloor+a_{v_{s-1}}\ge t, &\\ 	
           	\lfloor\frac{a_{j}}{\omega({j})}\rfloor+\lfloor\frac{a_{v_s}}{\omega({v_s})}\rfloor+\lfloor\frac{a_{v_{s-1}}}{\omega({v_{s-1}})}\rfloor\ge t. &\\ 	
           \end{cases}
           \]
           By Lemma \ref{degree complex3},  $\{i,j\},\{v_{s-1},v_s\}\in \Delta_{\mathbf{a}}(I^{(t)})$ and  $\{v_s,j\},\{i,\ell\},\{j,\ell\}\notin \Delta_{\mathbf{a}}(I^{(t)})$ for all $\ell\in [n]\backslash\{i,j,v_{s-1},v_s\}$. These conditions imply that $\Delta_{\mathbf{a}}(I^{(t)})$ is disconnected, a contradiction. Therefore, $\Delta(G)$ is a matroid complex.

\medskip
{\it Case $3$}: $(j,v_{s-1}),(v_s,i),(i,v_{s-1})\in E(\D)$. We define $a_i=\omega(i)(\omega({v_s})+1)$, $a_j=\omega(j)$, $a_{v_{s}}=\omega({v_{s}})$,  $a_{v_{s-1}}=\omega({v_{s-1}})(a_i+a_j-a_{v_s})$, $t=a_i+a_j+1$ and $a_{\ell}=0$ for any $\ell\in[n]\backslash\{i,j,v_{s-1},v_s\}$. Then, the following system of inequalities holds:
           \[
           \begin{cases}
           	a_{v_{s}}+\lfloor\frac{a_{v_{s-1}}}{\omega({v_{s-1}})}\rfloor\le t-1, &\\ 	
           	a_i+a_j\le t-1, &\\
           	\lfloor\frac{a_{i}}{\omega({i})}\rfloor+\lfloor\frac{a_{v_{s-1}}}{\omega({v_{s-1}})}\rfloor\ge t, &\\ 	
           	\lfloor\frac{a_{i}}{\omega({i})}\rfloor+\lfloor\frac{a_{v_s}}{\omega({v_s})}\rfloor+\lfloor\frac{a_{v_{s-1}}}{\omega({v_{s-1}})}\rfloor\ge t, &\\ 	
           	\lfloor\frac{a_{j}}{\omega({j})}\rfloor+a_{v_s}+\lfloor\frac{a_{v_{s-1}}}{\omega({v_{s-1}})}\rfloor\ge t. &\\ 	
           \end{cases}
           \]
           By Lemma \ref{degree complex3},  $\{i,j\},\{v_{s-1},v_s\}\in \Delta_{\mathbf{a}}(I^{(t)})$ and  $\{v_s,j\},\{i,\ell\},\{j,\ell\}\notin \Delta_{\mathbf{a}}(I^{(t)})$ for all $\ell\in [n]\backslash\{i,j,v_{s-1},v_s\}$. These conditions imply that $\Delta_{\mathbf{a}}(I^{(t)})$ is disconnected, a contradiction. Therefore, $\Delta(G)$ is a matroid complex.

\medskip
{\it Case $4$}: $(v_{s-1},j),(v_s,i),(i,v_{s-1})\in E(\D)$. Fix an integer $k$ such that $$k \ge \max\{ (\omega(i)\omega({v_s}) + \omega(i) - \omega({v_s}))(\omega({v_{s-1}}) - 1), 1 \}.$$ 
Set $a_i=\omega(i)(\omega({v_s})+1)$, $a_j=k\omega({j})$, $a_{v_s}=\omega({v_s})$, $a_{v_{s-1}}=k+a_i-\omega({v_s})$, $t=k+a_i+1$ and $a_{\ell}=0$ for any $\ell\in[n]\backslash\{i,j,v_{s-1},v_s\}$. Then, the following system of inequalities holds:
           \[
           \begin{cases}
           	a_{v_{s}}+a_{v_{s-1}}\le t-1, &\\ 	
           	a_i+\lfloor\frac{a_{j}}{\omega(j)}\rfloor\le t-1, &\\
           	\lfloor\frac{a_{i}}{\omega(i)}\rfloor+a_{v_{s-1}}\ge t, &\\ 	
           	\lfloor\frac{a_{i}}{\omega({i})}\rfloor+\lfloor\frac{a_{j}}{\omega({j})}\rfloor+\lfloor\frac{a_{v_{s-1}}}{\omega({v_{s-1}})}\rfloor\ge t, &\\ 	
           	a_i+\lfloor\frac{a_{j}}{\omega({j})}\rfloor+\lfloor\frac{a_{v_{s}}}{\omega({v_{s}})}\rfloor\ge t. &\\ 	
           \end{cases}
           \]
            By Lemma \ref{degree complex3},  $\{i,j\},\{v_{s-1},v_s\}\in \Delta_{\mathbf{a}}(I^{(t)})$ and  $\{v_s,j\},\{v_s,\ell\},\{v_{s-1},\ell\}\notin \Delta_{\mathbf{a}}(I^{(t)})$  for all $\ell\in [n]\backslash\{i,j,v_{s-1},v_s\}$. These conditions imply that $\Delta_{\mathbf{a}}(I^{(t)})$ is disconnected, which is a contradiction. Therefore, $\Delta(G)$ is a matroid complex.

            In summary, all of the above cases demonstrate that $\Delta(G)$ is a matroid complex, and the lemma follows.
           \end{proof}

\medskip

For a monomial ideal $I$ of $R$ and $j\in [n]$, define $I[j] = I R[x_j^{-1}] \cap R$ as the localization of $I$ with respect to the variable $x_j$. Note that $I[j] = I : x_j^{\infty}$. The following two lemmas are obvious.
          
\begin{Lemma}\label{local}
Let $I$ and $J$ be two monomial ideals in $R$ and let $j \subseteq [n]$. Then
\begin{itemize}
\item[(1)] $(I \cap J)[j] = I[j] \cap J[j]$,
\item[(2)] $(I^t)[j] = (I[j])^t$ for all $t \geq 1$, and
\item[(3)] if $I$ is Cohen-Macaulay, then so is $I[j]$.
\end{itemize}
\end{Lemma}

\begin{Lemma}\label{symbolic} Let $I\subseteq R$ be a monomial ideal. Then $(I^{(t)})[j]=(I[j])^{(t)}$ for all $t\ge 1$ and $j\in[n]$.
\end{Lemma}

\begin{proof} Assume that
$$I = Q_1 \cap Q_2 \cap \cdots \cap Q_k \cap Q_{k+1} \cap \cdots \cap Q_m \cap Q_{m+1} \cap \cdots \cap Q_s$$
is an irredundant primary decomposition of $I$, where
\begin{itemize}
\item $\operatorname{ht}(Q_i) = \operatorname{ht}(I)$ for $i = 1, \ldots, m$; and $\operatorname{ht}(Q_i) > \operatorname{ht}(I)$ for $i = m + 1, \ldots, s$.
\item $x_j\in Q_i$ for $i = 1, \ldots, k$; and $x_j\notin Q_i$ for $i = k + 1, \ldots, m$.
\end{itemize}
Let $P_i = \sqrt{Q_i}$ for $i=1,\ldots,s$. Observe that, for each $i$, $Q_i[j] = Q_i$ if $x_j\notin P_i$, and $Q_i[j] = R$ if $x_j\in P_i$. Therefore,
$$I[j]=(\bigcap\limits_{i=k+1}^{m}Q_i)\cap(\bigcap\limits_{\substack{i \ge m+1 \\ x_j\notin P_i}}Q_i).$$

For any $i \geq m+1$ with $x_j\notin P_i$, there exists an $l$ satisfying $k+1 \leq l \leq m$ and $P_l \subseteq P_i$, since  $P_i$ is an embedded prime of $I$. It follows that  $(I[j])^{(t)}=\bigcap\limits_{i=k+1}^{m}Q_i^t$.

Since $I^{(t)} = \bigcap\limits_{i=1}^{m} Q_i^t$, we have $(I^{(t)})[j]=\bigcap\limits_{i=k+1}^{m}Q_i^t$. Therefore, $(I^{(t)})[j]=(I[j])^{(t)}$, as required.
\end{proof}
      
  For a monomial $f$ in $R$, its support, $\operatorname{supp}(f)$, is the set of all variables  that appear in $f$. In other words, $\operatorname{supp}(f) = \{ x_j  | x_j \text{ divides } f \}$.

\begin{Lemma}\label{ngxv} Let $\D$ be a weighted oriented graph.  For any $t\ge 1$, if $I(\D)^{(t)}$ is Cohen-Macaulay, then $I(\D\backslash N_{\D}[v])^{(t)}$ is also Cohen-Macaulay for all $v\in[n]$.
\end{Lemma}

\begin{proof} Let $I=I(\D)$ and $J=I[v]$.  We assume that $N_{\D}^{+}{(v)}=\{{i_1},{i_2},\ldots,{i_p}\}$ and $N_{\D}^{-}{(v)}=\{{i_{p+1}},{i_{p+2}},\ldots,{i_s}\}$. Thus $N_{G}{(v)}=\{{i_1},{i_2},\ldots,{i_s}\}$. Since
$$I(\D)=(x_{v}x_{i_1}^{\omega({i_1})},\ldots,x_{v}x_{i_p}^{\omega({i_p})},x_v^{\omega(v)}x_{i_{p+1}},\ldots,x_v^{\omega(v)}x_{i_{s}})+I(\D\backslash v),$$
we obtain $J=(x_{i_1}^{\omega({i_1})},\ldots,{i_p}^{\omega({i_p})},x_{i_{p+1}},\ldots,x_{i_{s}})+I(\D\backslash v)$.

Let $\D'$ be the graph obtained from $\D \backslash v$ by removing all edges directed to some vertex in $\{i_1,\ldots,i_s\}$. In particular, $\{i_1,\ldots,i_s\}\in\Delta(G')$, where $G'$ is the underlying graph of $\D'$. Let $Q=(x_{i_1}^{\omega({i_1})},\ldots,x_{i_p}^{\omega({i_p})},x_{i_{p+1}},\ldots,x_{i_{s}})$. Then $J=Q+I(\D')$ since $x_mx_{i_k}^{\omega({i_k})}\in Q$ for every $k\in[s]$ and $m\in N_{\D\backslash v}^{-}({i_k})$. By Lemma \ref{decomposition},  $I(\D') = \bigcap\limits_{C \in \Gamma(\D')} I_C$,  so  $J = \bigcap\limits_{C \in \Gamma(\D')} (Q+I_C)$.

For simplicity, we set $G_v=G\backslash N_G[v]$ and  $\D_v=\D\backslash N_{\D}[v]$. Since $I^{(t)}$ is Cohen-Macaulay, by Lemmas \ref{local} and \ref{symbolic},  $J^{(t)} = (I[v])^{(t)} = (I^{(t)})[v]$ is also Cohen-Macaulay. In particular, $\sqrt{J^{(t)}}=\sqrt{J}=(x_{i_1},\ldots,x_{i_s})+I(G_v)$ is Cohen-Macaulay by Lemma \ref{radical}. Since $\{{i_1},\ldots,{i_s}\} \cap V(G_v)=\emptyset$, we conclude that $I(G_v)$ is Cohen-Macaulay. In particular, $G_v$ is well-covered. This also implies that $\operatorname{ht}(J) = \operatorname{ht}(\sqrt{J}) = s + \beta(G_v)$, where $\beta(G_v)$ is the covering number of $G_v$.

\medskip

For any $C\in \Gamma(G_v)$, define
$$A(C) = \{i\in\{{i_1},\ldots,{i_s}\} \mid N_{G'}(i) \not \subseteq C\} \cup C.$$

We will now prove the following four claims:

\medskip

{\em Claim $1$}: $A(C) \in \Gamma(\D')$ and $\operatorname{ht}(J)=\operatorname{ht}(Q+I_{A(C)})$ for any $C\in \Gamma(G_v)$.

Indeed, since $\{{i_1},\ldots,{i_s}\} \in\Delta(G')$, we can deduce that $A(C)$ is a vertex cover of $G'$. Furthermore,  by the definition of $A(C)$, $L_3(A(C)) = \emptyset$, so $A(C)$ is a minimal vertex cover of $G'$. In particular, $A(C)\in \Gamma(\D')$. On the other hand, $|C| = \beta(G_v)$ since $G_v$ is well-covered. Thus 
$$\operatorname{ht}(Q+I_{A(C)})=s+|C|=s+\beta(G_v)=\operatorname{ht}(J),$$
and the claim follows.

\medskip

{\em Claim $2$}: $\Gamma(G_v) = \{C'\cap V(\D_v)\mid C' \in \Gamma(\D') \text{ such that }  \operatorname{ht}(J)=\operatorname{ht}(Q+I_{C'})\}.$

By Claim $1$, it suffices to prove the inclusion relation $\supseteq$. To prove this, for any $C'\in\Gamma(\D')$ with $\operatorname{ht}(J)=\operatorname{ht}(Q+I_{C'})$, we let $C=C'\cap V(G_v)$. Since $G_v$ is an induced subgraph of $G'$ and  $C'$ is a vertex cover of $G'$, $C$ is a vertex cover of $G_v$.
On the other hand, since $\operatorname{ht}(J)=\operatorname{ht}(Q+I_{C'})$,  we have  $s+\beta(G_v)=s+|C|$, and therefore, $|C|=\beta(G_v)$. Therefore,  $C$ is a minimal vertex cover of $G_v$, as claimed.

\medskip

{\em Claim $3$}: For any  $C'\in\Gamma(\D')$ such that  $\operatorname{ht}(J)=\operatorname{ht}(Q+I_{C'})$, we have $A(C'\cap V(G_v)) = C'$.

First, we  show that $C'$ is a minimal vertex cover of $G'$. We write $C'$ as $$C'=(C'\cap V(\D_v))\cup(C'\cap\{i_1,\ldots,i_s\}).$$
By Claim $2$,  $C=C'\cap V(G_v)$ is a minimal vertex cover of $G_v$,  thus $L_3(C)=\emptyset$. Consequently, $L_3(C') \cap C = \emptyset$. On the other hand, there are no edges in $\D'$  directed to some vertex in $\{{i_1},\ldots,{i_s}\}$, so
$\{{i_1},\ldots,{i_s}\} \cap L_3(C') =\emptyset$. This  implies that $L_3(C')=\emptyset$, therefore, $C'$ is a minimal vertex cover of $G'$. Next, we  show that $C' = A(C)$. Note that, for any $i\in \{i_1,\ldots,i_s\}$, we have
$i\in C'$ whenever $N_{G'}(i) \not\subseteq C$. This implies that $A(C)\subseteq C'$. Since $C'$ is a minimal vertex cover of $G'$,  $C'$ must equal $A(C)$, as claimed.

\medskip

{\em Claim $4$}: $I_{A(C)} = (x_j\mid j\in A(C)\cap \{i_1,\ldots,i_s\})+ I_C$ for all $C\in \Gamma(G_v)$.

Indeed, since there are no edges in $\D'$ directed from any vertex in $\{i_1,\ldots,i_s\}$ to a vertex in $V(G_v)$  and the fact that $\{i_1,\ldots,i_s\}\in \Delta(G')$, we obtain
$$L_1(C) = L_1(A(C))\cap V(G_v) \text{ and }\{i_1,\ldots,i_s\}\cap A(C) \subseteq L_1(A(C)).$$
Hence, the claim that these formulas. 

\medskip

 Since $J = \bigcap\limits_{C' \in \Gamma(\D')} (Q+I_{C'})$, we have
$$J^{(t)} = \bigcap\limits_{C' \in \Gamma(\D') \colon \operatorname{ht}(Q + I_{C'})=\operatorname{ht}(J)} (Q+I_{C'})^t.$$
For  any $C\in \Gamma(G_v)$,   let $Q_C = Q + (x_i\mid i\in A(C) \cap \{i_1,\ldots,i_s\})$.
Together with the four claims above, we deduce
\begin{equation}\label{EQS}
J^{(t)} = \bigcap\limits_{C \in \Gamma(G_v)} (Q_C + I_{C})^t.
\end{equation}

Now, let $x^{\mathbf{a}}$ be a monomial such that $\operatorname{supp}(x^{\mathbf{a}})\subseteq V(\D_v)$. Then, for any $C\in\Gamma(G_v)$, we have $x^{\mathbf{a}} \in (Q_C+I_C)^t$ if and only if $x^{\mathbf{a}} \in I_C^t$. Together with Equation $(\ref{EQS})$, this fact  yields 
\begin{align*}
\sqrt{J^{(t)}:x^{\mathbf{a}}}&=\sqrt{\bigcap\limits_{C \in \Gamma(G_v)}(Q_C+I_{C})^t:x^{\mathbf{a}}}\\
&=(x_{i_1},\ldots,x_{i_s})+\sqrt{\bigcap\limits_{C \in \Gamma(G_v)}I_C^t:x^{\mathbf{a}}}\\
&=(x_{i_1},\ldots,x_{i_s})+\sqrt{I(\D_v)^{(t)}:x^{\mathbf{a}}},
\end{align*}

This implies that if $x^{\mathbf{a}}\notin I(\D_v)^{(t)}$, then $(x_{i_1},\ldots,x_{i_s})+\sqrt{I(\D_v)^{(t)}:x^{\mathbf{a}}}$ is Cohen-Macaulay according to Lemma \ref{cm}, and so is $\sqrt{I(\D_v)^{(t)}:x^{\mathbf{a}}}$. Using Lemma \ref{cm} again, we conclude that $I(\D_v)^{(t)}$ is Cohen-Macaulay, as required.
      \end{proof}
  
Now we are ready to prove the first main result of this section.

\begin{Theorem}\label{cmsymbolic} Let $\D$ be a weighted oriented graph with an underlying graph $G$. Then $I(\D)^{(t)}$ is Cohen-Macaulay for all $t\geqslant 1$ if and only if $G$ is a disjoint union of complete graphs.
\end{Theorem}
\begin{proof} $(\Longrightarrow)$ Assume that $I(\D)^{(t)}$ is Cohen-Macaulay for all $t\ge 1$. We will prove that  $\Delta(G)$ is a matroid by induction on  $\alpha(G)$. If $\alpha(G)=1$, then $\dim(\Delta(G))=0$, so $\Delta(G)$ is  clearly a matroid.  If $\alpha(G)=2$, then $\Delta(G)$ is a matroid by Lemma \ref{a=2}.

Assume that $\alpha(G) \geqslant 3$,  $\dim(\Delta(G))=\alpha(G)-1 \geq 2$. Since $I(\D)^{(t)}$ is Cohen-Macaulay for all $t \geq 1$, its radical $I(G) = \sqrt{I(\D)^{(t)}}$ is also Cohen-Macaulay. Therefore, by Lemma \ref{cmcomplex}, $\Delta(G)$ is connected. Next, for any $i \in [n]$, let $\D' = \D\setminus N_{\D}[i]$ and $G' = G\setminus N_G[i]$. Since $I(G)$ is Cohen-Macaulay, $G$ is well-covered. By Lemma \ref{well-covered}, $\alpha(G')=\alpha(G)-1$. By Lemma \ref{ngxv}, $I(\D')^{(t)}$ is also Cohen-Macaulay for all $t\geqslant 1$. Note that  $G'$ is the underlying graph of $\D'$ and that $\alpha(G')=\alpha(G)-1$, by the induction hypothesis, $\Delta(G')$ is a matroid. Since $\Delta(G') = \lk_{\Delta(G)}(i)$, the simplicial complex $\Delta(G)$ is locally a matroid. Since $\Delta(G)$ is connected and $\dim(\Delta(G))\geqslant 2$, by Lemma \ref{locally-matroid}, we have $\Delta(G)$ is a matroid, as desired.

 According to Lemma \ref{complete-graph}, $G$ is a disjoint union of complete graphs.

$(\Longleftarrow)$ Assume that $G$ is a disjoint union of complete graphs, say $G_1,\ldots,G_s$. Let $\D_i$ be the induced subgraph of $\D$ on $V(G_i)$ for $i=1,\ldots,s$. Thus the underlying graph of $\D_i$ is just $G_i$, and
$$I(\D) = I(\D_1)+\cdots+I(\D_s).$$

By \cite[Corollary 4.8]{HHTT}, in order to prove that $I(\D)^{(t)}$ is Cohen-Macaulay for every $t\geqslant 1$, it suffices to show that $I(\D_i)^{(t)}$ is Cohen-Macaulay for every $t\geqslant 1$. Therefore, we can assume that $G$ is a complete graph with $n$ vertices. In this case,
$$\Gamma(G) = \{[n]\setminus \{i\}\mid i=1,\ldots,n\}.$$
This together  with Lemma \ref{SPowers} yields
$$(x_1,\ldots,x_n)\notin \Ass(R/I(\D)^{(t)}).$$
In particular, $\depth(R/I(\D)^{(t)})\geqslant 1$. On the other hand,
$$\depth(R/I(\D)^{(t)}) \leqslant\dim (R/I\left(\D)^{(t)}\right) = \dim \left(R/\sqrt{I(\D)^{(t)}}\right)= \dim \left(R/I(G)\right)= 1.$$
This implies that $\depth(R/I(\D)^{(t)})=\dim(R/I(\D)^{(t)})$. Therefore, $I(\D)^{(t)}$ is Cohen-Macaulay, and the proof is now  complete.
\end{proof}

The following example shows that there is a weighted oriented  graph $\D$ such that $I(\D)^{(t)}$ is Cohen-Macaulay for many $t$, but not for all $t$.

\begin{Example} \label{E4} Let $k$ be a  positive integer and let $\D$ be an oriented graph with a vertex set $V(\D) = \{1,2,3,4\}$ and an edge set $E(\D) = \{(1,2),(2,3),(3,4)\}$. The weight function is
$$\omega(1)=\omega(4) = 1 \text{\ and\ } \omega(2)=\omega(3)=k.$$
For this graph, $I(\D) = (x_1x_2^k, x_2x_3^k, x_3x_4)$. Then,
$I(\D)^{(t)}$ is Cohen-Macaulay for all $t \leqslant k$ but $I(\D)^{(t)}$ is not Cohen-Macaulay for all $t >k$.
\end{Example}
\begin{proof} Since $I=I(\D)$ has a minimal primary decomposition as
$$I = (x_1,x_3)\cap (x_2^k,x_3) \cap (x_2,x_4)\cap (x_1,x_3^k,x_4)\cap (x_2^k,x_3^k,x_4).$$
We obtain
$$I^{(t)} = (x_1,x_3)^t\cap (x_2^k,x_3)^t\cap (x_2,x_4)^t.$$

By Lemma \ref{cm},  $I^{(t)}$ is not Cohen-Macaulay if and only if there is a monomial $f = x_1^{a_1}x_2^{a_2}x_3^{a_3}x_4^{a_4}\notin I^{(t)}$ such that
$$\sqrt{I^{(t)}\colon f} = (x_1,x_3)\cap (x_2,x_4),$$
this means that $f\in (x_2^k,x_3)^t$ but $f\notin (x_1,x_3)^t$ and $f\notin (x_2,x_4)^t$. Equivalently,
$$
\begin{cases}
\lfloor\frac{a_2}{k}\rfloor+a_3\geqslant t, &\\ 	
a_1+a_3\leqslant t-1, &\\
a_2+a_4 \leqslant t-1.
\end{cases}
$$
We can verify that this system has a solution $(a_1,a_2,a_3,a_4)\in \NN^4$ if and only if $t > k$, as required.
\end{proof}

We will conclude this section with a corollary of Theorem \ref{cmsymbolic}.

\begin{Corollary}\label{cmordinary} Let $\D$ be a weighted oriented graph with an underlying graph $G$. Then $I(D)^t$ is Cohen-Macaulay for all $t\ge 1$ if and only if $G$ is a disjoint union of edges.
\end{Corollary}

\begin{proof} $(\Longrightarrow)$ Assume that $I(G)^t$ is Cohen-Macaulay for all $t\geqslant 1$. Then, for every $t\geqslant 1$,  $I(\D)^t$ is unmixed. Therefore,  $I(\D)^t = I(\D)^{(t)}$. In particular, $I(\D)^{(t)}$ is Cohen-Macaulay.  According to  Theorem \ref{cmsymbolic}, $G$ is a disjoint union of complete graphs.

On the other hand, since $I(\D)^t = I(\D)^{(t)}$ for all $t\geqslant 1$, $G$ is bipartite by \cite[Theorem 3.3]{GMV}. Therefore,  every connected component of $G$ is an edge, consequently,  $G$ consists of  disjoint edges.

$(\Longleftarrow)$ Now, assume that $G$ is a disjoint union of edges, i.e., that  $I(\D)$ is a complete intersection. In this case,  using  \cite[Corollaries 3.7 and 4.8]{HHTT}, we conclude that $I(\D)^t$ is Cohen-Macaulay for all $t\geqslant 1$.
\end{proof}
        
\section{Equality of ordinary and symbolic powers of edge ideals} \label{sec:equality}

Let $\D$ be a weighted oriented graph. In this section, we will characterize the Cohen-Macaulayness of $I(\D)^t$ for a given integer $t\geqslant 2$. To do this, we will  need an operation on monomial ideals. Let $R = K[x_1, \ldots, x_n]$ be a polynomial ring of $n$ variables $x_1,\ldots,x_n$ over a field $K$.

 For any weight vector $\mathbf{w} = (w_1, \ldots, w_n) \in \mathbb{Z}_{>0}^n$, we define the weighted action of $\mathbf{w}$ on  $x^{\mathbf{a}}$ as $\mathbf{w}(x^{\mathbf{a}}) = x_1^{w_1 a_1} \cdots x_n^{w_n a_n}$.	For a monomial ideal $I\subseteq R$, define ${\bf w}(I)$ to be the monomial ideal
$${\bf w}(I) = (w(m) \mid m \text{ is a monomial in } I)$$
and
$$\mathcal{G}({\bf w}(I))=\{{\bf w}(m)\mid m\in\mathcal{G}(I)\}.$$

The following lemma follows immediately from this definition.
\begin{Lemma}\label{w1} Let $I,J\subseteq R$ be monomial ideals. Then,
\begin{itemize}
\item[(1)] $I=J$ if and only if ${\bf w}(I) = {\bf w}(J)$,
\item[(2)] ${\bf w}(I^t) = ({\bf w}(I))^t$ for all $t\geqslant 1$,  and
\item[(3)] ${\bf w}(I^{(t)}) = ({\bf w}(I))^{(t)}$ for all $t\geqslant 1$.
\end{itemize}
\end{Lemma}
		
\begin{Lemma}\label{betti} Let $I\subseteq R$ be a monomial ideal. Then, $\beta_{i,\mathbf{a}}(R/I)=\beta_{i,{\bf w}(\mathbf{a})}(R/{\bf w}(I))$ for all $i\geqslant 0$ and $\mathbf{a}\in \mathbb{N}^n$, where ${\bf w}(\mathbf{a})=(w_1a_1,\ldots,w_na_n)$.
\end{Lemma}
	
\begin{proof} By Lemma \ref{betti1}, it suffices to prove that $K^{\mathbf{a}}(I) = K^{\mathbf{w}(\mathbf{a})}(\mathbf{w}(I))$, where  $K^{\mathbf{a}}(I)$  is the (upper) Koszul simplicial complex of $I$ in degree $\mathbf{a}$.  First, we prove the inclusion relation $K^{\mathbf{a}}(I) \subseteq K^{\mathbf{w}(\mathbf{a})}(\mathbf{w}(I))$. For any  squarefree vector $\boldsymbol{\tau} \in K^{\mathbf{a}}(I)$,   $x^{\mathbf{a} - \boldsymbol{\tau}} \in I$. There exists  some $h \in \mathcal{G}(I)$ such that  $h\mid x^{\mathbf{a} - \boldsymbol{\tau}}$. It is obvious  that $\mathbf{w}(h)\mid \mathbf{w}(x^{\mathbf{a} - \boldsymbol{\tau}})$. Note that $\mathbf{w}(x^{\mathbf{a} - \boldsymbol{\tau}}) = x^{\mathbf{w}(\mathbf{a}) - \mathbf{w}(\boldsymbol{\tau})}$ and $x^{\mathbf{w}(\mathbf{a}) - \mathbf{w}(\boldsymbol{\tau})}\mid x^{\mathbf{w}(\mathbf{a}) - \boldsymbol{\tau}}$, so $\mathbf{w}(h)\mid x^{\mathbf{w}(\mathbf{a}) - \boldsymbol{\tau}}$. This implies that $\boldsymbol{\tau} \in K^{\mathbf{w}(\mathbf{a})}(\mathbf{w}(I))$, and the inclusion follows.

For the converse inclusion, suppose that  $\boldsymbol{\tau}\in K^{{\bf w}(\mathbf{a})}({\bf w}(I))$, then there exists an $x^\mathbf{b}\in \mathcal{G}(I)$ such that ${\bf w}(x^\mathbf{b})\mid x^{{\bf w}(\mathbf{a})-\boldsymbol{\tau}}$. Since ${\bf w}(x^\mathbf{b})=x^{{\bf w}(\mathbf{b})}$, we  can conclude  that $w_ib_i\le w_ia_i-\tau_i$ for all $i\in[n]$. Thus, for each $i\in [n]$, we have the following system of inequalities:
			\[
			\begin{cases}
				b_i\leqslant a_i-1, &\text{if $\tau_i=1$,}\\ 	
				b_i\leqslant a_i, &\text{if $\tau_i=0$.}
			\end{cases}
			\]
Consequently, $x^{\mathbf{b}}\mid  x^{\mathbf{a}-\boldsymbol{\tau}}$, and thus  $\boldsymbol{\tau}\in K^\mathbf{a}(I)$, as required.
\end{proof}
	
\begin{Lemma}\label{cm2} Let $I$ be a monomial ideal. Then, $\pd(R/I) = \pd(R/\wb(I))$. In particular, $I$ is Cohen-Macaulay if and only if ${\bf w}(I)$ is also Cohen-Macaulay.
\end{Lemma}
\begin{proof} Let $p = \pd(R/\wb(I))$. Then, $\beta_{p,\bb}(R/\wb(I)) \ne 0$ for some $\bb\in\NN^n$. By Lemma \ref{betaneq0},  we have $x^{\bb} = \lcm(\wb(m_1),\ldots,\wb(m_s))$ for some $m_1,\ldots,m_s\in \mathcal{G}(I)$. Let $x^{\ab}=\lcm(m_1,\ldots,m_s)$, then $\bb = \wb(\ab)$. By Lemma \ref{betti},  $\beta_{p,\bb}(R/\wb(I)) = \beta_{p,\ab}(R/I)$. Therefore, $\beta_{p,\ab}(I)\ne 0$ and  $\pd(R/I) \geqslant \pd(R/\wb(I))$. We can similarly prove the reverse inequality and conclude that  $\pd(R/I) = \pd(R/\wb(I))$.

Finally, note that $\dim(R/I) = \dim(R/\wb(I))$. Together with the equality $$\pd(R/I) = \pd(R/\wb(I)),$$
 we see that $I$ is Cohen-Macaulay if and only if so is $\wb(I)$, as required.
\end{proof}

We will  now study the Cohen-Macaulaynes of $I(\D)^t$ for some $t\geqslant 2$. To do so, we first characterize the equality $I(\D)^t = I(\D)^{(t)}$. For the edge ideal of a simple graph, this equality is given by \cite{RTY}.

\begin{Lemma}{\em (\cite[Lemma 3.10]{RTY})}\label{G}  Let $t\ge 2$ be an integer and let $I(G)$ be the edge ideal of a graph $G$. Then, the following conditions are equivalent:
\begin{itemize}
\item[(1)] $G$ contains no odd cycles of length $2s - 1$, where $2 \leqslant s \leqslant t$.
\item[(2)] $I(G)^{(t)} = I(G)^t$.
\end{itemize}
\end{Lemma}

We extend this result to weighted oriented graphs as follows:

\begin{Theorem}\label{equal} Let $t\ge 2$ be an integer and let $\D$ be a weighted oriented graph with an underlying graph $G$. Then  the following conditions are equivalent:
			\begin{itemize}
				\item[(1)] $I(\D)^t=I(\D)^{(t)}$.
				\item[(2)] Every vertex in $V^{+}(\D)$ is a sink, and $G$ contains no odd cycles of length $2s - 1$, where $2 \leqslant s \leqslant t$.
			\end{itemize}
\end{Theorem}
\begin{proof} (2) $\implies$ (1):  If every vertex in $V^{+}(D)$ is a sink, then $I(D) = \mathbf{w}(I(G))$, where $\mathbf{w} = (\omega(1), \omega(2), \ldots, \omega(n))$. Since $G$ contains no odd cycles of length $2s-1$, where  $2 \leqslant s \leqslant t$, then, by Lemma \ref{G}, $I(G)^{(t)} = I(G)^t$. Together with Lemma \ref{w1}, this yields $I(\D)^t = I(\D)^{(t)}$.
	     	
(1) $\implies$ (2):  First, we  prove that every vertex in $V^{+}(\D)$ is a sink. Suppose by contradiction that $v\in V^{+}(\D)$ is not a sink. Then there exist $k,u\in V(\D)$ such that  $(u,v),(v,k)\in E(\D)$. Set $f=(x_ux_v^{\omega(v)})^{t-1}x_k^{\omega(k)}$.  Due to Lemma \ref{SPowers}, we  first prove that $f\in I(\D)^{(t)}$, or equivalently, $f\in I_C^t$ for every  $C\in \Gamma(G)$. For any $C\in \Gamma(G)$, if $k\in C$, then $x_k^{\omega(k)}\in I_C$. Since  $(x_ux_v^{\omega(v)})^{t-1}\in I(\D)^{t-1}\subseteq I_C^{t-1}$, it follows that $f\in I_C^t$. If $k\notin C$, then $v\in L_1(C)$ and $x_v^{t-1}\in I_C^{t-1}$. Note that $x_vx_k^{\omega(k)}\in I(D)$, it follows that $f\in I_C^t$. Therefore, $f\in I(D)^{(t)}$.
	     	
Next, we prove that  $f\notin I(D)^t$. Suppose by contradiction that $f\in I(D)^t$,  we will consider the following three cases:

{\it Case $1$}: $(k,u)\in E(\D)$. We express $f$ as  $f=h(x_ux_v^{\omega(v)})^{t_1}(x_vx_k^{\omega(k)})^{t_2}(x_kx_u^{\omega(u)})^{t_3}$, where $h$ is a monomial, and $t_1+t_2+t_3=t$ with each $t_i\ge 0$. By comparing the degrees of each variable in the expression of  $f$, we obtain
	     	\[
	     	\begin{cases}
	     		 t-1\ge t_1+t_3\omega(u),&\\
(t-1)\omega(v)\ge t_1\omega(v)+t_2, &\\ 	
\omega(k)\ge t_2\omega(k)+t_3.
	     		\end{cases}
	     	\]
From the expression   $\omega(k)\ge t_2\omega(k)+t_3$, we can deduce that  $t_2$ is either $0$ or $1$. If $t_2=0$, then $t-1\geqslant t_1+t_3\omega(u)\ge t_1+t_3=t$, a contradiction. Thus $t_2=1$, which forces  $t_3=0$ and $t_1=t-1$. Then $(t-1)\omega(v)\ge t_1\omega(v)+t_2=(t-1)\omega(v)+1$, a contradiction.

{\it Case $2$}: $(u,k)\in E(\D)$. We express $f$ as  $f=h(x_ux_v^{\omega(v)})^{t_1}(x_vx_k^{\omega(k)})^{t_2}(x_ux_k^{\omega(k)})^{t_3}$, where $h$ is a monomial, and $t_1+t_2+t_3=t$ with each $t_i\ge 0$. By comparing the degrees of each variable in the expression of  $f$, we obtain
	     	\[
	     	\begin{cases}
	     		t-1\geqslant t_1+t_3, &\\ 	
	     		(t-1)\omega(v)\geqslant t_1\omega(v)+t_2, &\\ 	
	     	    \omega(k)\ge t_2\omega(k)+t_3\omega(k).
	     	\end{cases}
	     	\]
From the expression  $\omega(k)\geqslant t_2\omega(k)+t_3\omega(k)$, we can conclude that $t_2+t_3\leqslant 1$. If $t_2+t_3=0$, then $t_1=t$. Substituting this into the original expression yields  $t-1\ge t_1+t_3=t$, which is a contradiction. Thus $t_2+t_3=1$. If $t_2=1$ and $t_3=0$, then $t_1=t-1$, which gives us the inequality $(t-1)\omega(v)\ge t_1\omega(v)+t_2=(t-1)\omega(v)+1$, a contradiction.  Therefore, $t_2=0$ and $t_3=1$, which implies that $t_1=t-1$ and  $t-1\geqslant t_1+t_3=t$, a contradiction.

{\it Case $3$}: $\{u,k\}\notin E(G)$. We express $f$ as  $f=h(x_ux_v^{\omega(v)})^{t_1}(x_vx_k^{\omega(k)})^{t_2}$, where $h$ is a monomial, and $t_1+t_2=t$ with each $t_i\ge 0$. 
By comparing the degrees of each variable in the expression of  $f$, we obtain
	     \[
	     	\begin{cases}
	     		t-1\ge t_1, &\\
	     		(t-1)\omega(v)\geqslant t_1\omega(v)+t_2, &\\ 	
	     		\omega(k)\geqslant t_2\omega(k).
	     	\end{cases}
	     	\]
From the expression  $\omega(k)\geqslant t_2\omega(k)$, we can deduce that $t_2\leqslant 1$. If $t_2=0$, then $t-1\ge t_1=t$, a contradiction. Thus, $t_2=1$. This yields  $(t-1)\omega(v)\ge t_1\omega(v)+t_2=(t-1)\omega(v)+1$, a contradiction.

In summary, all three cases lead to a contradiction, meaning  $f\notin I(\D)^t$. However, $f \in I(\D)^{(t)}$. Therefore, $I(\D)^t \ne I(\D)^{(t)}$, a contradiction. Consequently, every vertex in $V^{+}(\D)$ is a sink, and  $I(\D)=\wb(I(G))$, where $\wb=(\omega(1),\omega(2),\ldots,\omega(n))$. Since $I(\D)^t=I(\D)^{(t)}$,  $I(G)^t=I(G)^{(t)}$ by Lemma \ref{w1}. By Lemma \ref{G}, $G$ contains no odd cycles of length $2s - 1$, where $2 \leqslant s \leqslant t$. Thus,  the theorem follows.
\end{proof}

We now characterize the Cohen-Macaulayness of $I(\D)^t$ for $t\geqslant 3$, thereby improving upon Corollary \ref{cmordinary}.

\begin{Theorem}\label{cmPowers} Let $\D$ be a weighted oriented graph with an underlying graph $G$. Then, the following conditions are equivalent:
\begin{itemize}
\item[(1)] $I(\D)^t$ is Cohen-Macaulay for every $t\geqslant 1$.
\item[(2)] $I(\D)^t$ is Cohen-Macaulay for some $t\geqslant 3$.
\item[(3)] $G$ is a disjoint union of edges.
\end{itemize}
\end{Theorem}
\begin{proof}
According to Corollary \ref{cmordinary}, statements $(1)$ and $(3)$ are equivalent.  Since $(1)$  implies $(2)$ trivially, it suffices to prove that $(2)$ implies $(3)$. Assume that $I(\D)^t$ is Cohen-Macaulay for some $t\ge 3$, then $I(\D)^{(t)}=I(\D)^t$. By Theorem \ref{equal}, any vertex in $V^{+}(D)$ is a sink. Thus ${\bf w}(I(G))=I(\D)$, where $\wb =(\omega(1),\omega(2),\ldots, \omega(n))$. By Lemma \ref{w1},  $I(G)^t$ is Cohen-Macaulay. Therefore, $I(G)$ is a complete intersection  by \cite[Theorem 3.8]{RTY}. This implies that $G$ is just a disjoint union of edges, as required.
\end{proof}

Finally, we characterize  the Cohen-Macaulayness of $I(\D)^2$.

\begin{Theorem}\label{cmPower2} Let $\D$ be a weighted oriented graph with an underlying graph $G$. Then $I(\D)^2$ is Cohen-Macaulay if and only if the following two conditions hold:
\begin{itemize}
\item[(1)] Every vertex in $V^+(\D)$ is a sink, and
\item[(2)] $G$ is a triangle-free graph in the class $W_2$.
\end{itemize}
\end{Theorem}
\begin{proof} Let $\wb=(\omega(1),\omega(2),\ldots,\omega(n))$. If $I(\D)^2$ is Cohen-Macaulay, then $I(\D)^2=I(\D)^{(2)}$. According to  Theorem \ref{equal}, every vertex in $V^{+}(D)$ is a sink. In particular,  $\wb(I(G))=I(\D)$. Therefore, by 	Lemmas \ref{w1} and  \ref{cm2}, $I(G)^2$ is Cohen-Macaulay. Together with \cite[Theorem 4.4]{HT2}, this implies that the condition (2) holds.

Now, assume that two conditions (1) and (2) are true. From the condition (1), we have  $\wb(I(G))=I(\D)$. From the condition (2),  we obtain that $I(G)^2$ is Cohen-Macaulau by
\cite[Theorem 4.4]{HT2}. Therefore, $I(\D)^2$ is Cohen-Macaulay by 	Lemmas \ref{w1} and  \ref{cm2}, and the proof is complete.
\end{proof}

\medskip
\hspace{-6mm} {\bf Acknowledgments}
		
		\vspace{3mm}
		\hspace{-6mm}  The fourth author is  supported by the Natural Science Foundation of Jiangsu Province (No. BK20221353) and the National Natural Science Foundation of China (12471246). The third author is partially supported by Vietnam National Foundation for Science and Technology Development (Grant \#101.04-2024.07). The main part of this
work was done during the third author's visit to Soochow University in Suzhou, China.  He would like to express his gratitude to Soochow University for its warm hospitality.

		\medskip
		\hspace{-6mm} {\bf Data availability statement}
		
		\vspace{3mm}
		\hspace{-6mm}  The data used to support the findings of this study are included within the article.
		
		\medskip
		\hspace{-6mm} {\bf Conflict of interest}
		
		\vspace{3mm}
		\hspace{-6mm}  The authors declare that they have no competing interests.

\end{document}